\documentclass[11pt]{amsart}
\usepackage[english]{babel}
\usepackage[T1]{fontenc}

\usepackage[margin=0cm,top=2.7cm, inner=2.7cm,outer=2.7cm,bottom=2.7cm]{geometry}
\usepackage[dvips]{graphicx}

\usepackage{amsmath,amscd,amsthm,amsfonts,latexsym,epsfig}

\theoremstyle{plain}
\newtheorem{theorem}                 {Theorem}      [section]
\newtheorem{proposition}  [theorem]  {Proposition}
\newtheorem{corollary}    [theorem]  {Corollary}
\newtheorem{lemma}        [theorem]  {Lemma}
\newtheorem*{theorem*}{Theorem}
\theoremstyle{definition}
\newtheorem{example}      [theorem]  {Example}
\newtheorem{remark}       [theorem]  {Remark}

\def \r{\mbox{${\mathbb R}$}}

\def \c{\mbox{${\mathbb C}$}}
\def \h{\mbox{${\mathbb H}$}}

\def \s{\mbox{${\mathbb S}$}}
\def \z{\mbox{${\mathbb Z}$}}
\def \g{\mbox{$g_{\varepsilon}$}}
\def \n{\mbox{${{\nabla}}^\varepsilon$}}

\DeclareMathOperator{\cst}{constant}

\DeclareMathOperator{\trace}{trace}

\begin{document}

\title{Helix surfaces in the Berger Sphere}

\author{Stefano Montaldo}

\address{Universit\`a degli Studi di Cagliari\\
Dipartimento di Matematica e Informatica\\
Via Ospedale 72\\
09124 Cagliari}
\email{montaldo@unica.it}

\author{Irene I. Onnis}
\address{Departamento de Matem\'{a}tica\\ ICMC/USP-Campus de S\~ao Carlos\\
Caixa Postal 668\\ 13560-970 S\~ao Carlos, SP, Brazil}
\email{onnis@icmc.usp.br}

\thanks{The first author was supported by PRIN 2010--11, {\it Variet\`a reali e complesse: geometria, topologia e analisi armonica}, Italy, N. 2010NNBZ78\_003. The second author was supported by CNPq, Brazil.}

\begin{abstract}
We characterize helix surfaces in the Berger sphere, that is surfaces which form a constant angle with the Hopf vector field. In particular, we show that,
 locally, a helix surface is determined by a suitable 1-parameter family of isometries of the Berger sphere and by a geodesic of a $2$-torus in the $3$-dimensional sphere.
 \end{abstract}

\maketitle

\section{Introduction}

In the euclidean space  $\r^3$ a {\em helix surface}  or a {\em constant angle surface} is an oriented surface such that its normal vector field forms a constant angle with a fixed direction in the space. These surfaces have an important role in the physics of interfaces in liquid crystals and of layered fluids, as shown by Cermelli and Di Scala
in \cite{CDS07}, where they have also obtained a remarkable relation with a Hamilton-Jacobi type equation. \\

A constant direction in the euclidean space can be thought as a Killing vector field with constant norm. Thus, if we want to generalize the notion of helix surfaces in a 3-dimensional Riemannian manifold, a natural problem is to
study surfaces such that their normal vector field forms a constant angle with a Killing  vector field of constant norm. \\

Among the three dimensional manifolds, beside the space forms, probably the most important are the 3-dimensional homogeneous manifolds.  Most of these spaces admit a Killing vector field of constant norm and thus it is natural to study the corresponding helix surfaces. In fact, the study and classification of helix surfaces in 3-dimensional homogeneous manifolds was done: for surfaces in  $\s^2\times\r$ by Dillen--Fastenakels--Van der Veken--Vrancken (\cite{DM09}); for surfaces in  $\h^2\times\r$ by Dillen--Munteanu  (\cite{DM09}); for surfaces in  the Heisenberg group by Fastenakels--Munteanu--Van Der Veken  (\cite{FMV11});  for surfaces in  $Sol_3$ by L\'opez--Munteanu (\cite{LM11}). \\

We also would like to point out that, in a similar way, we can define helix submanifolds in
 higher dimensional euclidean  spaces  and we advise the interested reader to have a look at \cite{DSRH10,DSRH09,RH11}.\\

In this paper, in order to take a step towards the  classification of helix surfaces in 3-dimensional homogeneous manifolds, we consider surfaces in the 3-dimensional Berger sphere which is defined, using the Hopf fibration, as follows.
Let
$\s^2({1}/{2})=\{(z,t)\in\c\times\r\colon |z|^2+t^2={1}/{4}\}\subset\c\times\r$ be the usual  2-sphere and let $\s^3=\{(z,w)\in\c^2\colon |z|^2+|w|^2=1\}\subset\c\times\c=\r^4$ be the usual 3-sphere. Then the Hopf map $\psi:\s^3\to\s^2({1}/{2})$, defined by
$$
\psi(z,w)=\frac{1}{2}\,(2z\bar{w},|z|^2-|w|^2)\,,
$$
is a Riemannian submersion and the vector fields
$$
X_1(z,w)=(iz,iw),\quad X_2(z,w)=(-i\bar{w},i\bar{z}),\quad X_3(z,w)=(-\bar{w},\bar{z})\,,
$$
parallelize $\s^3$ with $X_1$  being vertical and $X_2$, $X_3$ horizontal. The vector field $X_1$ is called the {\em Hopf vector field}. The Berger sphere $\s^3_\varepsilon$, $\varepsilon>0$, is the sphere $\s^3$ endowed with the metric
$$
g_\varepsilon(X,Y)=\langle X,Y\rangle+(\varepsilon^2-1)\langle X,X_1\rangle \,\langle Y,X_1\rangle\,,
$$
where $\langle,\rangle$ represents the canonical metric of $\s^3$. \\

The Hopf vector field $X_1$ is a Killing vector field of constant norm, thus we define a {\em helix surface} in  $\s^3_\varepsilon$ as a surface
 such that its unit normal $N$ satisfies
$$
|g_{\varepsilon}( X_1,N)|=\varepsilon \cos\vartheta,
$$
for fixed $\vartheta\in[0,\pi/2]$.
 \vspace{3mm}

From a classical result of Reeb \cite{Reeb}, a compact surface in the Berger sphere cannot be transverse to the Hopf vector field everywhere. This means that the notion of helix surfaces in $\s^3_\varepsilon$, with $\vartheta\neq\pi/2$, is meaningful only in the non compact case. For this reason our study will be local and it will aim to the following characterization of helix surfaces which represents the main result of the paper.\vspace{3mm}

\noindent {\bf Theorem}~{\bf \ref{teo-principal}.}
{\em Let $M^2$ be a helix surface in the Berger sphere $\s^3_\varepsilon$ with constant angle $\vartheta\neq\pi/2$. Then there exist local coordinates on $M^2$ such that the position vector of $M^2$ in $\r^4$  is
$$
F(u,v)=A(v)\,\beta(u)\,,
$$
where
$$
\beta(u)=(\sqrt{c_1}\,\cos (\alpha_1 u),\sqrt{c_1}\,\sin (\alpha_1 u),\sqrt{c_2}\,\cos (\alpha_2 u),\sqrt{c_2}\,\sin (\alpha_2 u))
$$
is a geodesic in the torus $\s^1(\sqrt{c_1})\times\s^1(\sqrt{c_2})\subset \s^3$, with
$$
c_{1,2}=\frac{1}{2}\mp \frac{\varepsilon\cos\vartheta}{2 \sqrt{B}}\,,\quad \alpha_1=\frac{2B}{\varepsilon} c_2\,,\quad \alpha_2=\frac{2B}{\varepsilon} c_1\,,\quad B=1+(\varepsilon^2-1)\cos^2\vartheta\,,
$$
while $A(v)=A(\xi,\xi_1,\xi_2,\xi_3)(v)$ is a 1-parameter family of $4\times 4$ orthogonal matrices commuting with a complex structure $J_1$ of $\r^4$, as described in \eqref{eq-descrizione-A}, with $\xi=\cst$ and
$$
  \cos^2(\xi_1(v)) \,\xi_{2}'(v)-\sin^2(\xi_1(v))\, \xi_{3}'(v)=0\,.
$$
Conversely, a parametrization $F(u,v)=A(v)\,\beta(u)$, with $\beta(u)$ and $A(v)$ as above, defines
 a helix surface in the Berger sphere $\s^3_\varepsilon$ with constant angle $\vartheta\neq\pi/2$. }
\vspace{3mm}

\section{Helix surfaces}

With respect to the orthonormal  basis on $\s^3_{\varepsilon}$ defined by
\begin{equation}\label{eq-basis}
    E_1=\varepsilon^{-1}\,X_1,\quad
     E_2=X_2,\quad
     E_3=X_3,
\end{equation}
the Levi-Civita connection $\n$ of $(\s^3_{\varepsilon},g_\varepsilon)$ is given by:
\begin{equation}
\begin{array}{lll}
\n_{E_{1}}E_{1}=0,& \n_{E_{2}}E_{2}=0,& \n_{E_{3}}E_{3}=0,\\
\n_{E_{1}}E_{2}=\varepsilon^{-1}(2-\varepsilon^2)E_{3}, & &\n_{E_{1}}E_{3}=-\varepsilon^{-1}(2-\varepsilon^2)E_{2},\\
 \n_{E_{2}}E_{1}=-\varepsilon E_{3},& \n_{E_{3}}E_{1}=\varepsilon E_{2},&
 \n_{E_{3}}E_{2}=-\varepsilon E_{1}=-\n_{E_{2}}E_{3}.
\end{array}
\label{nabla}
\end{equation}
Let $M^2$ be an oriented helix surface in $\s^3_\varepsilon$ and let  $N$ be a unit normal vector field. Then, by definition,
$$
|g_{\varepsilon}(E_1,N)|=\cos\vartheta,
$$
for fixed $\vartheta\in[0,\pi/2]$. Note that $\vartheta\neq 0$. In fact, if it were then the vector fields $E_2$ and $E_3$   would be tangent to the surface $M^2$, which is absurd since  the horizontal distribution of the Hopf map is not integrable. If $\vartheta=\pi/2$, we have that $E_1$ is always tangent to $M$ and, therefore, $M$ is a Hopf tube. Therefore, from now on we assume that the constant angle $\vartheta\neq \pi/2,0$.\\

The Gauss and Weingarten formulas, for all $X,Y\in C(TM)$, are
\begin{equation}\label{gauss-wein}\begin{aligned}
    \n_X Y&=\nabla_X Y+\alpha(X,Y),\\
    \n_X N&=-A(X),
    \end{aligned}
\end{equation}
 where with $A$ we have indicated  the shape operator of $M$ in $\s^3_\varepsilon$, with $\nabla$ the induced Levi-Civita connection on $M$ and by $\alpha$  the second fundamental form on $M$ in $\s^3_\varepsilon$.
Decomposing  $E_1$ into its tangent and normal components we have
$$
E_1=T+\cos\vartheta\, N\,,
$$
where $T$ satisfies $\g(T,T)=\sin^2\vartheta.$\\

For all $X\in C(TM)$, we have that
\begin{equation}\label{eq1}
\begin{aligned}
    \n_X E_1&=\n_X T-\cos\vartheta\,A(X)\\&=\nabla_X T+g_{\varepsilon}( A(X),T)\,N-\cos\vartheta\,A(X).
    \end{aligned}
\end{equation}
On the other hand (we refer to \cite{B}), if $X=\sum X_i E_i$,
\begin{equation}\label{eq2}\begin{aligned}
    \n_X E_1&=\varepsilon\,(X_3 E_2-X_2 E_3)\\
    &=\varepsilon\,\g(  JX,T)\,N-\varepsilon\,\cos\vartheta JX,
    \end{aligned}
\end{equation}
where $JX$ denotes the rotation of angle $\pi/2$ on $TM$.
Identifying the tangent and normal components of \eqref{eq1} and \eqref{eq2} respectively, we obtain
\begin{equation}\label{eq3}
    \nabla_X T= \cos\vartheta\,(A(X)-\varepsilon\,JX)
\end{equation}
and
\begin{equation}\label{eq4}
    \g(  A(X)-\varepsilon\,JX,T)=0.
\end{equation}

\begin{lemma}\label{princ}
Let $M^2$ be an oriented helix surface with constant angle $\vartheta$  in $\s^3_\varepsilon$. Then, we have that:
\begin{itemize}
  \item[(i)] with respect to the basis $\{T,JT\}$, the matrix associated to the shape operator $A$ takes the following form
 $$
 A=\left(
\begin{array}{cc}
 0 & -\varepsilon \\
 -\varepsilon & \lambda\\
 \end{array}
 \right)\,,
 $$
 for some function $\lambda$ on $M$;
 \item[(ii)] the Levi-Civita connection $\nabla$ of $M$ is given by
$$
 \nabla_T T=-2\varepsilon\cos\vartheta\, JT,\qquad \nabla_{JT} T=\lambda\cos\vartheta\, JT\,,
 $$
$$
\nabla_T JT=2\varepsilon\cos\vartheta\, T,\qquad \nabla_{JT} JT=-\lambda\cos\vartheta\, T\,;
$$
\item[(iii)] the Gauss curvature of $M$ is constant and satisfies
$$
K=4(1-\varepsilon^2)\,\cos^2\vartheta\,;
$$
\item[(iv)] the function $\lambda$, defined in \rm{(i)}, satisfies the following equation
\begin{equation}\label{lambda}
T\lambda+\lambda^2\,\cos\vartheta+4(\varepsilon^2-1)\,\cos^3\vartheta+4\cos\vartheta=0\,.
\end{equation}
\end{itemize}
\end{lemma}
\begin{proof} Point (i) follows directly from \eqref{eq4}.
From \eqref{eq3} and using
$$
\g(  T,T)=\g(  JT,JT)=\sin^2\vartheta,\quad \g(  T,JT)=0\,,
$$
we obtain (ii). From the Gauss equation of $M$ in $\s^3_\varepsilon$ (see \cite{B}), and taking into account  (i), we have that the Gauss curvature of $M$ is given by
$$
\begin{aligned}K&=\det A+\varepsilon^2+4(1-\varepsilon^2)\,\cos^2\vartheta\\
&=4(1-\varepsilon^2)\,\cos^2\vartheta\,.
\end{aligned}
$$
Finally, \eqref{lambda} follows from the Codazzi equation (see \cite{B}):
$$
\nabla_X A(Y)-\nabla_Y A(X)-A[X,Y]=4(1-\varepsilon^2)\,\cos\vartheta\,(\g(  Y,T) X-\g(  X,T) Y)\,,
$$
 putting $X=T$, $Y=JT$ and using (ii).

\end{proof}

\begin{remark}
We point out that if a helix surface is minimal, that is $\trace A=0$, then $\vartheta=\pi/2$. In fact,
from (i) of Lemma~\ref{princ}, $\lambda=0$ and using \eqref{lambda} it follows that
$\cos\vartheta(1+(\varepsilon^2-1)\cos^2\vartheta)=0$, which implies, since $1+(\varepsilon^2-1)\cos^2\vartheta$ is different from zero for all $\vartheta\in[0,\pi/2]$, that $\vartheta=\pi/2$.
\end{remark}

Now, as $g_\varepsilon (E_1,N)=\cos\vartheta$, there exists a smooth function $\varphi$ on $M$ so that
\begin{equation*}\label{eq-def-N-Ei}
N=\cos\vartheta E_1+\sin\vartheta\cos\varphi\,E_2+\sin\vartheta\sin\varphi\,E_3.
\end{equation*}
Therefore
\begin{equation}\label{eq:def-T}
T=E_1-\cos\vartheta\,N=\sin\vartheta\,[\sin\vartheta\,E_1-\cos\vartheta\cos\varphi\,E_2-\cos\vartheta\sin\varphi\,E_3]
\end{equation}
and
$$
JT=\sin\vartheta\,(\sin\varphi\,E_2-\cos\varphi\,E_3)\,.
$$
Also
\begin{equation}\begin{aligned}\label{eqTJ}
 A(T)&=-\n_T N=(T\varphi+\varepsilon^{-1}(2-\varepsilon^2)\,\sin^2\vartheta+\varepsilon\cos^2\vartheta)\,JT,\\
A(JT)&=-\n_{JT} N=(JT\varphi)\,JT-\varepsilon\,T\,.
\end{aligned}
\end{equation}
Comparing \eqref{eqTJ} with (i) of  Lemma~\ref{princ}, it results that
\begin{equation}\label{eqTJ1}
\left\{\begin{aligned}
JT\varphi&=\lambda\,,\\
T\varphi&=-2\varepsilon^{-1}\,B\,,
\end{aligned}
\right.
\end{equation}
where
\begin{equation}\label{eq-defB}
B=1+(\varepsilon^2-1)\cos^2\vartheta\,.
\end{equation}
We observe that, as
$$
[T,JT]=\cos\vartheta\,(2\varepsilon\, T-\lambda\,JT),
$$
the compatibility condition of system~\eqref{eqTJ1},
$$
(\nabla_T JT-\nabla_{JT} T)\varphi=[T,JT]\varphi=T(JT\varphi)-JT(T\varphi),
$$
is equivalent to \eqref{lambda}.\\

We now choose local coordinates $(u,v)$ on $M$ such that
\begin{equation}\label{eq:local-coordinates}
\partial_u=T.
\end{equation}
 Also, as $\partial_v$ is tangent to $M$, it can be written in the form
 \begin{equation}\label{eq-definition-Fv}
 \partial_v=a\,T+b\,JT\,,
 \end{equation}
for certain functions $a=a(u,v)$ and $b=b(u,v)$. Since
$$
0=[\partial_u,\partial_v]=(a_u+2\varepsilon b\cos\vartheta)\,T+(b_u-b\lambda\cos\vartheta)\,JT\,,
$$
we obtain
\begin{equation}\label{eqab}
\left\{\begin{aligned}
a_u&=-2\varepsilon b\cos\vartheta\,,\\
b_u&=b\lambda\cos\vartheta\,.
\end{aligned}
\right.
\end{equation}
Moreover, writing \eqref{lambda}  as
$$
\lambda_u+\cos\vartheta\,\lambda^2+4(\varepsilon^2-1)\,\cos^3\vartheta+4\cos\vartheta=0\,,
$$
after integration, one gets
\begin{equation}\label{eqlambda}
\lambda(u,v)=2\,\sqrt{B}\tan (\eta(v)-2\cos\vartheta \sqrt{B}\,u)\,,
\end{equation}
for some smooth function $\eta$ depending on $v$.
Replacing \eqref{eqlambda} in  \eqref{eqab} and solving the system, we obtain
\begin{equation}\label{solab}
\left\{\begin{aligned}
a(u,v)&=\frac{\varepsilon}{\sqrt{B}}\sin (\eta(v)-2\cos\vartheta \sqrt{B}\,u),\\
b(u,v)&=\cos(\eta(v)-2\cos\vartheta \sqrt{B}\,u)\,.
\end{aligned}
\right.
\end{equation}
Therefore \eqref{eqTJ1} becomes
\begin{equation}\label{eqTJ2}\left\{\begin{aligned}
\varphi_u&=-2\varepsilon^{-1}B\,,\\
\varphi_v&=0\,,
\end{aligned}
\right.
\end{equation}
of which the general solution is given by
\begin{equation}
\varphi(u,v)=-2\varepsilon^{-1}B\,u+c\,,
\end{equation}
where $c$ is a real constant.

With respect to the local coordinates $(u,v)$ described above we have the following characterization of the position vector of a helix surface.
\begin{proposition}
Let $M^2$ be a helix surface in $\s^3_\varepsilon$ with constant angle $\vartheta$.  Then, with respect to the local coordinates $(u,v)$ on $M$ defined in \eqref{eq:local-coordinates},  the position vector $F$ of $M^2$ in $\r^4$ satisfies the equation
\begin{equation}\label{eqquarta}
\frac{\partial^4F}{\partial u^4}+(\tilde{b}^2-2\tilde{a})\,\frac{\partial^2F}{\partial u^2}+\tilde{a}^2\,F=0\,,
\end{equation}
where
\begin{equation}\label{eq:value-a-b}
\tilde{a}=\varepsilon^{-2}\sin^2\vartheta\, B\,, \qquad \tilde{b}=-2 \varepsilon^{-1}\, B
\end{equation}
and $B=1+(\varepsilon^2-1)\cos^2\vartheta$.
\end{proposition}
\begin{proof}
Let $M^2$ be a helix surface in $\s_{\varepsilon}^3\subset\r^4$ and let $F$ be the position vector of $M^2$ in $\r^4$. Then, with respect to the local coordinates $(u,v)$ on $M$ defined in \eqref{eq:local-coordinates}, we can write $F(u,v)=(F_1(u,v),\dots,F_4(u,v))$. By definition, taking into account \eqref{eq:def-T}, we have that
\begin{equation}\label{eq-F2-as-Ei}
\begin{aligned}\partial_u F&=(\partial_uF_1,\partial_uF_2,\partial_uF_3,\partial_uF_4)=T\\
&=\sin\vartheta\,[\sin\vartheta\,{E_1}_{|F(u,v)}-\cos\vartheta\cos\varphi\,{E_2}_{|F(u,v)}-\cos\vartheta\sin\varphi\,{E_3}_{|F(u,v)}]\,.
\end{aligned}
\end{equation}
Using the expression of $E_1$, $E_2$ and $E_3$ with respect to the coordinates vector fields of $\r^4$, the latter implies that
\begin{equation}\label{eqprime}\left\{\begin{aligned}
\partial_uF_1&=\sin\vartheta\,(-\varepsilon^{-1}\sin\vartheta\,F_2+\cos\vartheta\cos\varphi\,F_4+\cos\vartheta\sin\varphi\,F_3)\,,\\
\partial_uF_2&=\sin\vartheta\,(\varepsilon^{-1}\sin\vartheta\,F_1+\cos\vartheta\cos\varphi\,F_3-\cos\vartheta\sin\varphi\,F_4)\,,\\
\partial_uF_3&=-\sin\vartheta\,(\varepsilon^{-1}\sin\vartheta\,F_4+\cos\vartheta\cos\varphi\,F_2+\cos\vartheta\sin\varphi\,F_1)\,,\\
\partial_uF_4&=\sin\vartheta\,(\varepsilon^{-1}\sin\vartheta\,F_3-\cos\vartheta\cos\varphi\,F_1+\cos\vartheta\sin\varphi\,F_2)\,.\\
\end{aligned}
\right.
\end{equation}
Moreover, taking the derivative with respect to $u$ of \eqref{eqprime} we find two constants $\tilde{a}$ and $\tilde{b}$ such that
\begin{equation}\label{eqsegunda}\left\{\begin{aligned}
(F_1)_{uu}&=\tilde{a}\,F_1+\tilde{b}\,(F_2)_u\,,\\
(F_2)_{uu}&=\tilde{a}\,F_2-\tilde{b}\,(F_1)_u\,,\\
(F_3)_{uu}&=\tilde{a}\,F_3+\tilde{b}\,(F_4)_u\,,\\
(F_4)_{uu}&=\tilde{a}\,F_4-\tilde{b}\,(F_3)_u\,,\\
\end{aligned}
\right.
\end{equation}
where, using \eqref{eqTJ2},
$$
\tilde{a}=-\frac{\varepsilon^{-1}\sin^2\vartheta}{2}\varphi_u=\varepsilon^{-2}\sin^2\vartheta B, \qquad \tilde{b}=\varphi_u=-2 \varepsilon^{-1}\, B\,.
$$
Finally, taking twice the derivative of \eqref{eqsegunda} with respect to $u$ and using  \eqref{eqprime}--\eqref{eqsegunda} in the derivative we obtain the desired equation \eqref{eqquarta}.
\end{proof}
Integrating \eqref{eqquarta}, we have the following
\begin{corollary}\label{cor-Fuv}
Let $M^2$ be a helix surface in $\s^3_\varepsilon$ with constant angle $\vartheta$.  Then, with respect to the local coordinates $(u,v)$ on $M$ defined in \eqref{eq:local-coordinates},  the position vector $F$ of $M^2$ in $\r^4$
is given by
$$
F(u,v)=\cos(\alpha_1 u)\,g^1(v)+\sin(\alpha_1 u)\,g^2(v)+\cos(\alpha_2 u)\,g^3(v)+\sin(\alpha_2 u)\,g^4(v),
$$
where
$$
\alpha_{1,2}=\frac{1}{\varepsilon}(B\pm\varepsilon \sqrt{B} \cos\vartheta)
$$
are real constants, while the $g^i(v)$, $i\in\{1,\dots,4\}$, are mutually orthogonal vector fields  in $\r^4$, depending only on $v$, such that
$$
g_{11}=\langle g^1(v),g^1(v)\rangle=g_{22}=\langle g^2(v),g^2(v)\rangle=\frac{\varepsilon}{2B} \alpha_2\,,
$$
$$
g_{33}=\langle g^3(v),g^3(v)\rangle=g_{44}=\langle g^4(v),g^4(v)\rangle=\frac{\varepsilon}{2B} \alpha_1\,.
$$
\end{corollary}
\begin{proof}
Since \eqref{eqquarta} can be seen as an ODE in $u$ with constant coefficients, a direct integration, using the characteristic polynomial and taking into account that the integration constants depend on $v$, gives the solution
$$
F(u,v)=\cos(\alpha_1 u)\,g^1(v)+\sin(\alpha_1 u)\,g^2(v)+\cos(\alpha_2 u)\,g^3(v)+\sin(\alpha_2 u)\,g^4(v)\,,
$$
where
$$
\alpha_{1,2}=\sqrt{\frac{\tilde{b}^2-2\tilde{a}\pm\sqrt{\tilde{b}^4-4\tilde{a}\tilde{b}^2}}{2}}
$$
are two constants, while the $g^i(v)$, $i\in\{1,\dots,4\}$,  are vector fields in $\r^4$ which depend only on $v$. Now, taking into account the values of $\tilde{a}$ and $\tilde{b}$ given in \eqref{eq:value-a-b}, we obtain
$$
\alpha_{1,2}=\frac{1}{\varepsilon}(B\pm\varepsilon \sqrt{B} \cos\vartheta)\,.
$$
Next, since $|F|^2=1$ and using \eqref{eqquarta}, \eqref{eqprime} and \eqref{eqsegunda}  we find that the position vector $F(u,v)$  and its derivatives  must satisfy the relations:
\begin{equation}\label{eq:Fprocuct}
\begin{array}{lll}
\langle  F,F\rangle=1\,,& \langle F_u,F_u\rangle=\varepsilon^{-2}B\sin^2\vartheta\,,& \langle F,F_u\rangle=0\,,\\
\langle F_u,F_{uu}\rangle=0\,,&\langle F_{uu},F_{uu}\rangle=D\,,&\langle F,F_{uu}\rangle=-\varepsilon^{-2}B\sin^2\vartheta,\\
\langle F_u,F_{uuu}\rangle=-D\,,&\langle F_{uu},F_{uuu}\rangle=0\,,&\langle F,F_{uuu}\rangle=0\,,\\
\langle F_{uuu},F_{uuu}\rangle=E,& \qquad &
\end{array}
\end{equation}
where
$$
D=\varepsilon^{-2}\,B\,\tilde{b}^2\sin^2\vartheta-3\tilde{a}^2\,,\qquad E=(\tilde{b}^2-2\tilde{a})\,D-\varepsilon^{-2}\,B\,\tilde{a}^2\sin^2\vartheta\,.
$$
Putting  $g_{ij}(v)=\langle g^i(v),g^j(v)\rangle$, and evaluating  the relations \eqref{eq:Fprocuct}  in $(0,v)$, we obtain:
\begin{equation}\label{uno}
    g_{11}+g_{33}+2g_{13}=1\,,
\end{equation}
\begin{equation}\label{due}
    \alpha_1^2\,g_{22}+\alpha_2^2\,g_{44}+2\alpha_1\alpha_2\,g_{24}=\varepsilon^{-2}B\sin^2\vartheta\,,
\end{equation}
\begin{equation}\label{tre}
    \alpha_1\,g_{12}+\alpha_2\,g_{14}+\alpha_1\,g_{23}+\alpha_2\,g_{34}=0\,,
\end{equation}
\begin{equation}\label{quatro}
    \alpha_1^3\,g_{12}+\alpha_1\alpha_2^2\,g_{23}+\alpha_1^2\alpha_2\,g_{14}+\alpha_2^3g_{34}=0\,,
\end{equation}
\begin{equation}\label{cinque}
    \alpha_1^4\,g_{11}+\alpha_2^4\,g_{33}+2\alpha_1^2\alpha_2^2\,g_{13}=D\,,
\end{equation}
\begin{equation}\label{sei}
    \alpha_1^2\,g_{11}+\alpha_2^2\,g_{33}+(\alpha_1^2+\alpha_2^2)\,g_{13}=\varepsilon^{-2}B\sin^2\vartheta\,,
\end{equation}
\begin{equation}\label{sette}
    \alpha_1^4\,g_{22}+\alpha_1^3\alpha_2\,g_{24}+\alpha_1\alpha_2^3\,g_{24}+\alpha_2^4\,g_{44}=D\,,
\end{equation}
\begin{equation}\label{otto}
    \alpha_1^5\,g_{12}+\alpha_1^3\alpha_2^2\,g_{23}+\alpha_1^2\alpha_2^3\,g_{14}+\alpha_2^5\,g_{34}=0\,,
\end{equation}
\begin{equation}\label{nove}
    \alpha_1^3\,g_{12}+\alpha_1^3\,g_{23}+\alpha_2^3\,g_{14}+\alpha_2^3\,g_{34}=0\,,
\end{equation}
\begin{equation}\label{dieci}
    \alpha_1^6\,g_{22}+\alpha_2^6\,g_{44}+2\alpha_1^3\alpha_2^3\,g_{24}=E\,.
\end{equation}

From \eqref{tre}, \eqref{quatro}, \eqref{otto}, \eqref{nove}, it follows that
$$
g_{12}=g_{14}=g_{23}=g_{34}=0\,.
$$
Also, from  \eqref{uno}, \eqref{cinque} and \eqref{sei}, we obtain
$$
g_{11}=\frac{\varepsilon^2\,(D+\alpha_2^4)-2B\sin^2\vartheta\,\alpha_2^2}{\varepsilon^2(\alpha_1^2-\alpha_2^2)^2}\,,\qquad g_{13}=0\,,\qquad
g_{33}=\frac{\varepsilon^2\,(D+\alpha_1^4)-2B\sin^2\vartheta\,\alpha_1^2}{\varepsilon^2(\alpha_1^2-\alpha_2^2)^2}.
$$
Moreover, using \eqref{due}, \eqref{sette} and \eqref{dieci}, we get
$$
g_{22}=\frac{\varepsilon^2\,(E-2D\alpha_2^2)+B\sin^2\vartheta\,\alpha_2^4}{\varepsilon^2\alpha_1^2\,(\alpha_1^2-\alpha_2^2)^2}\,,\quad g_{24}=0\,,\qquad
g_{44}=\frac{\varepsilon^2\,(E-2D\alpha_1^2)+B\sin^2\vartheta\,\alpha_1^4}{\varepsilon^2\alpha_2^2\,(\alpha_1^2-\alpha_2^2)^2}\,.
$$
Finally, a long but straightforward computation gives
$$
g_{11}=g_{22}=\frac{\varepsilon}{2B} \alpha_2\,,\qquad g_{33}=g_{44}=\frac{\varepsilon}{2B} \alpha_1
\,.
$$
\end{proof}

\section{The main result}
We are now in the right position to state the main result of the paper. Before doing this we recall that, looking at $\s^3_{\varepsilon}$ in $\r^4$,
its isometry group can be identified with:
$$
\{A\in \mathrm{O}(4)\colon A J_{1}=\pm  J_{1}A\}\,,
$$
where $ J_{1}$ is the complex structure of $\r^4$ defined by
$$
 J_{1} =\begin{pmatrix}
 0&-1 & 0 & 0 \\
 1&0 & 0 & 0 \\
 0&0 & 0 & -1 \\
  0&0 & 1 & 0 \\
 \end{pmatrix},
 $$
while $\mathrm{O}(4)$ is the orthogonal group (see, for example, \cite{t}). Suppose now we are given a 1~-~parameter family $A(v)\,, v\in(a,b)\subset\r$, consisting of $4\times 4$ orthogonal matrices commuting with $J_1$. In order to describe explicitly the family $A(v)$, we shall use two others complex structures of $\r^4$, namely
$$
 J_{2} =\begin{pmatrix}
 0&0 & 0 & -1 \\
 0&0 & -1 & 0 \\
 0&1& 0 & 0 \\
  1&0 & 0 & 0 \\
 \end{pmatrix}\,,\qquad
 J_{3} =\begin{pmatrix}
 0&0 & -1 &0 \\
 0&0 & 0 & 1 \\
 1&0& 0 & 0 \\
0&-1 & 0 & 0 \\
 \end{pmatrix}\,.
 $$
Since $A(v)$ is an orthogonal matrix the first row must be a unit vector ${\mathbf r}_1(v)$ of $\r^4$ for all $v\in(a,b)$. Thus, with out loss of generality, we can take
$$
{\mathbf r}_1(v)=(\cos\xi_1(v)\cos\xi_2(v), -\cos\xi_1(v)\sin\xi_2(v), \sin\xi_1(v)\cos\xi_3(v),-\sin\xi_1(v)\sin\xi_3(v))\,,
$$
for some real functions $\xi_1,\xi_2$ and $\xi_3$ defined in $(a,b)$. Since $A(v)$ commutes with $J_1$ the second row of $A(v)$ must be ${\mathbf r}_2(v)=J_{1}{\mathbf r}_1(v)$. Now, the four vectors $\{{\mathbf r}_1, J_1{\mathbf r}_1, J_2{\mathbf r}_1,J_3{\mathbf r}_1\}$ form an orthonormal basis of $\r^4$, thus the third row ${\mathbf r}_3(v)$ of $A(v)$ must be a linear combination of them. Since ${\mathbf r}_3(v)$ is unit and it is orthogonal to both ${\mathbf r}_1(v)$ and $J_1{\mathbf r}_1(v)$, there exists a function $\xi(v)$ such that
$$
{\mathbf r}_3(v)=\cos\xi(v) J_2 {\mathbf r}_1(v)+\sin\xi(v) J_3 {\mathbf r}_1(v)\,.
$$
Finally the fourth row of $A(v)$ is ${\mathbf r}_4(v)=J_1{\mathbf r}_3(v)=-\cos\xi(v) J_3 {\mathbf r}_1(v)+\sin\xi(v) J_2 {\mathbf r}_1(v)$.
This means that any $1$-parameter family $A(v)$ of $4\times 4$ orthogonal matrices commuting with $J_1$
can be described by four functions $\xi_1,\xi_2,\xi_3$ and $\xi$ as
\begin{equation}\label{eq-descrizione-A}
A(\xi,\xi_1,\xi_2,\xi_3)(v)=
\begin{pmatrix}
{\mathbf r}_1(v)\\
J_1{\mathbf r}_1(v)\\
\cos\xi(v) J_2 {\mathbf r}_1(v)+\sin\xi(v) J_3 {\mathbf r}_1(v)\\
-\cos\xi(v) J_3 {\mathbf r}_1(v)+\sin\xi(v) J_2 {\mathbf r}_1(v)
\end{pmatrix}\,.
\end{equation}

\begin{theorem}\label{teo-principal}
Let $M^2$ be a helix surface in the Berger sphere $\s^3_\varepsilon$ with constant angle $\vartheta\neq\pi/2$. Then, locally, the position vector of $M^2$ in $\r^4$, with respect to the local coordinates $(u,v)$ on $M$ defined in \eqref{eq:local-coordinates}, is
$$
F(u,v)=A(v)\,\beta(u)\,,
$$
where
\begin{equation}\label{eq-beta-main-theorem}
\beta(u)=(\sqrt{g_{11}}\,\cos (\alpha_1 u),\sqrt{g_{11}}\,\sin (\alpha_1 u),\sqrt{g_{33}}\,\cos (\alpha_2 u),\sqrt{g_{33}}\,\sin (\alpha_2 u))
\end{equation}
is a geodesic in the torus $\s^1(\sqrt{g_{11}})\times\s^1(\sqrt{g_{33}})\subset \s^3$, with $g_{11}$, $g_{33}$, $\alpha_1$, $\alpha_2$ the four constants given in Corollary~\ref{cor-Fuv}, and $A(v)=A(\xi,\xi_1,\xi_2,\xi_3)(v)$ is a 1-parameter family of $4\times 4$ orthogonal matrices commuting with $J_1$, as described in \eqref{eq-descrizione-A}, with $\xi=\cst$ and
\begin{equation}\label{eq-alpha123}
  \cos^2(\xi_1(v)) \,\xi_{2}'(v)-\sin^2(\xi_1(v))\, \xi_{3}'(v)=0\,.
\end{equation}
Conversely, a parametrization $F(u,v)=A(v)\,\beta(u)$, with $\beta(u)$ and $A(v)$ as above, defines
 a helix surface in the Berger sphere $\s^3_\varepsilon$ with constant angle $\vartheta\neq\pi/2$.
\end{theorem}
\begin{proof}
With respect to the local coordinates $(u,v)$ on $M$ defined in \eqref{eq:local-coordinates}, Corollary~\ref{cor-Fuv} implies that the position vector of a  helix  surface in $\r^4$ is given by
$$
F(u,v)=\cos(\alpha_1 u)\,g^1(v)+\sin(\alpha_1 u)\,g^2(v)+\cos(\alpha_2 u)\,g^3(v)+\sin(\alpha_2 u)\,g^4(v)\,,
$$
where the vector fields  $\{g^i(v)\}$ are mutually orthogonal  and
$$
||g^1(v)||=||g^2(v)||=\sqrt{g_{11}}=\text{constant}\,,
$$
$$
||g^3(v)||=||g^4(v)||=\sqrt{g_{33}}=\text{constant}\,.
$$
Thus, if we put $e_i(v)=g^i(v)/||g^i(v)||$, $i\in\{1,\dots,4\}$, we can write:
\begin{equation}\label{eq:Fei}
F(u,v)=\sqrt{g_{11}}\,(\cos (\alpha_1\,u)\,e_1(v)+\sin(\alpha_1\,u)\,e_2(v))  +\sqrt{g_{33}}\,(\cos (\alpha_2\,u)\,e_3(v)+\sin(\alpha_2\,u)\,e_4(v)\,.
\end{equation}
Denote by $\bar{J}$ the $4\times 4$ matrix with entries
$\bar{J}_{i,j}=\langle  J_{1} e_i,e_j\rangle$, $i,j=1,\ldots,4$. We shall prove that
$\bar{J}=( J_{1})^{T}$. For this, since
$$
 J_{1}F(u,v)={X_1}_{|F(u,v)}=\varepsilon\,{E_1}_{|F(u,v)}=\varepsilon\,(F_u+\cos\vartheta\,N),
$$
and using \eqref{eqquarta}--\eqref{eq:Fprocuct},  we obtain the following identities
\begin{equation}\label{eq-fu-jf-main}
\begin{aligned}
&\langle J_{1}F,F_u\rangle=\varepsilon^{-1}\sin^2\vartheta,\\&\langle J_{1}F,F_{uu}\rangle=0\,,\\
&\langle F_u, J_{1}F_{uu}\rangle=\varepsilon^{-3}B\sin^2\vartheta\,(\sin^2\vartheta-2B):=I\,,\\
&\langle J_{1}F_u,F_{uuu}\rangle=0\,,\\
&\langle J_{1}F_u,F_{uu}\rangle+\langle J_{1}F,F_{uuu}\rangle=0\,,\\
&\langle J_{1}F_{uu},F_{uuu}\rangle+\langle J_{1}F_u,F_{uuuu}\rangle=0\,.
\end{aligned}
\end{equation}
Indeed, since $g_{\varepsilon}(F_u,E_1)=\varepsilon\langle F_u, J_1 F\rangle$ and, from \eqref{eq-F2-as-Ei}, $g_{\varepsilon}(F_u,E_1)=\sin^2\vartheta$ we obtain the first  of \eqref{eq-fu-jf-main}, from which, deriving with respect to $u$, it follows the second and the fifth.
Also,  \eqref{eqsegunda} implies
$$
\langle F_u, J_{1}F_{uu}\rangle=\tilde{a}\,\langle J_{1}F,F_u\rangle+\tilde{b}\,|F_u|^2=I.
$$
The fourth of \eqref{eq-fu-jf-main} is a consequence of the fifth, in view of
$$\langle J_{1}F,F_{uuuu}\rangle=-\langle J_1 F,(\tilde{b}^2-2\tilde{a})\,F_{uu}+\tilde{a}^2\,F \rangle=0.$$
Finally deriving the fourth of \eqref{eq-fu-jf-main} with respect to $u$ we obtain the last equation of \eqref{eq-fu-jf-main}.\\

Evaluating \eqref{eq-fu-jf-main} in  $(0,v)$, they become respectively:
\begin{equation}\label{eq1bis}
    \alpha_1 g_{11}\langle J_{1}e_1,e_2\rangle+\alpha_2\,g_{33}\langle J_{1}e_3,e_4\rangle+\sqrt{g_{11}g_{33}}\,(\alpha_1\langle J_{1}e_3,e_2\rangle+\alpha_2\langle J_{1}e_1,e_4\rangle)
    =\varepsilon^{-1}\sin^2\vartheta,
\end{equation}
\begin{equation}
     \langle J_{1}e_1,e_3\rangle=0\,,
\end{equation}
\begin{equation}\label{39}
\alpha_1^3\,g_{11}\langle J_{1}e_1,e_2\rangle+\alpha_2^3\,g_{33}\langle J_{1}e_3,e_4\rangle+\sqrt{g_{11}g_{33}}\,(\alpha_1\alpha_2^2\langle J_{1}e_3,e_2\rangle+\alpha_1^2\alpha_2\langle J_{1}e_1,e_4\rangle)=-I,
\end{equation}
\begin{equation}
    \langle J_{1}e_2,e_4\rangle=0\,,
\end{equation}
\begin{equation}\label{eq2bis}
    \alpha_1\langle J_{1}e_2,e_3\rangle+\alpha_2\langle J_{1}e_1,e_4\rangle=0\,,
\end{equation}
\begin{equation}\label{eq3bis}
    \alpha_2\langle J_{1} e_2,e_3\rangle+\alpha_1\langle J_{1}e_1,e_4\rangle=0\,.
\end{equation}
We point out that  to obtain the previous identities we have divided by $\alpha_1^2-\alpha_2^2=4 \varepsilon^{-1} \sqrt{B^3 \cos^2\vartheta}$ which is, by the assumption on $\vartheta$, always different from zero.
From \eqref{eq2bis} and \eqref{eq3bis}, taking into account that $\alpha_1^2-\alpha_2^2\neq 0$, it results that
\begin{equation}\label{eq4bis}
     \langle J_{1}e_3,e_2\rangle=0\,,\qquad \langle J_{1}e_1,e_4\rangle=0\,.
\end{equation}
Therefore
$$
|\langle J_{1}e_1,e_2\rangle|=1=|\langle J_{1}e_3,e_4\rangle|.
$$
Substituting \eqref{eq4bis} in \eqref{eq1bis} and \eqref{39}, we obtain the system
\begin{equation}\label{j34}
\left\{\begin{aligned}
&\alpha_1 g_{11}\langle J_{1}e_1,e_2\rangle+\alpha_2\,g_{33}\langle J_{1}e_3,e_4\rangle=\varepsilon^{-1}\sin^2\vartheta\\
 & \alpha_1^3 g_{11}\langle J_{1}e_1,e_2\rangle+\alpha_2^3\,g_{33}\langle J_{1}e_3,e_4\rangle=-I\,,
\end{aligned}
\right.
\end{equation}
a solution of which is
$$
\langle J_{1}e_1,e_2\rangle=\frac{\varepsilon I+\alpha_2^2\sin^2\vartheta}{\varepsilon g_{11}\,\alpha_1 (\alpha_2^2-\alpha_1^2)}\,,\qquad \langle J_{1}e_3,e_4\rangle=-\frac{\varepsilon I+\alpha_1^2\sin^2\vartheta}{\varepsilon g_{33}\,\alpha_2(\alpha_2^2-\alpha_1^2)}\,.
$$
Now, as
$$
g_{11}\,g_{33}=\frac{\sin^2\vartheta}{4B}\,,\qquad \alpha_1\,\alpha_2=\frac{B}{\varepsilon^2}\sin^2\vartheta\,,\qquad (\alpha_1^2-\alpha_2^2)^2=\frac{16B^3}{\varepsilon^{2}}\cos^2\vartheta\,,
$$
it results that
$$
\langle J_{1}e_1,e_2\rangle\langle J_{1}e_3,e_4\rangle=1\,.
$$
Moreover, a direct check shows that $\langle  J_{1}e_1,e_2\rangle>0$.
Consequently,
$\langle  J_{1}e_1,e_2\rangle=\langle  J_{1}e_3,e_4\rangle=1$.
We have thus proved that $\bar{J}=( J_{1})^{T}$. \\

If we fix the canonical orthonormal basis  of $\r^4$ given by
$$
\bar{e}_1=(1,0,0,0))\,,\quad \bar{e}_2=(0,1,0,0)\,,\quad \bar{e}_3=(0,0,1,0)\,,\quad \bar{e}_4=(0,0,0,1)\,,
$$
there must exists a 1-parameter family of $4\times 4$ orthogonal matrices $A(v)\in \mathrm{O}(4)$, with $ J_{1}A(v)=A(v) J_{1}$, such that $e_i(v)=A(v)\bar{e}_i$.
Replacing $e_i(v)=A(v)\bar{e}_i$ in \eqref{eq:Fei} we obtain
$$
F(u,v)=A(v)\beta(u)\,,
$$
where the curve
$$
\beta(u)=(\sqrt{g_{11}}\,\cos (\alpha_1 u),\sqrt{g_{11}}\,\sin (\alpha_1 u),\sqrt{g_{33}}\,\cos (\alpha_2 u),\sqrt{g_{33}}\,\sin (\alpha_2 u))\,,
$$
is clearly a geodesic of the torus $\s^1(\sqrt{g_{11}})\times\s^1(\sqrt{g_{33}})\subset \s^3$.\\

Let now examine the $1$-parameter family $A(v)$ that, according to \eqref{eq-descrizione-A},
depends on four functions $\xi_1(v),\xi_2(v),\xi_3(v)$ and $\xi(v)$. From \eqref{eq-definition-Fv}, and taking into account \eqref{solab}, it results that $\langle F_v, F_v\rangle=\sin^2 \vartheta=\cst$. The latter implies that
\begin{equation}\label{eq-fv-fv-sin-theta-d-u}
\frac{\partial}{\partial u}\langle F_v, F_v\rangle_{| u=0}=0\,.
\end{equation}
Now, if we denote by ${\mathbf c_1},{\mathbf c_2},{\mathbf c_3},{\mathbf c_4}$ the four colons of $A(v)$, \eqref{eq-fv-fv-sin-theta-d-u} implies that
\begin{equation}\label{sistem-c23-c24}
\begin{cases}
\langle {\mathbf c_2}',{\mathbf c_3}'\rangle=0\\
\langle {\mathbf c_2}',{\mathbf c_4}'\rangle=0\,,
\end{cases}
\end{equation}
where with $'$ we means the derivative with respect to $v$.
Replacing in \eqref{sistem-c23-c24} the expressions of the ${\mathbf c_i}$'s as functions of $\xi_1(v),\xi_2(v),\xi_3(v)$ and $\xi(v)$, we obtain
\begin{equation}\label{sistemK-H}
\begin{cases}
\xi'\, h(v)=0\\
\xi'\, k(v)=0\,,
\end{cases}
\end{equation}
where $h(v)$ and $k(v)$ are two functions such that
$$
h^2+k^2=4 (\xi_1')^2+\sin^2(2\xi_1)\, (-\xi'+\xi_2'+\xi_3')^2\,.
$$
Thus we have two possibilities:
\begin{itemize}
\item[(i)] $\xi=\cst$;
\item[] or
\item[(ii)] $4 (\xi_1')^2+\sin^2(2\xi_1)\, (-\xi'+\xi_2'+\xi_3')^2=0$.
\end{itemize}
We will show that case (ii) cannot occurs, more precisely we will show that if (ii) happens  than the parametrization $F(u,v)=A(v)\beta(u)$ defines a Hopf tube, that is the Hopf vector field $E_1$ is tangent to the surface. To this end let's compute the unit normal vector field $N$ to the parametrization
$F(u,v)=A(v) \beta(u)$. If we put
$$
F_u=g_{\varepsilon}(F_u,E_1)E_1+g_{\varepsilon}(F_u,E_2)E_2+g_{\varepsilon}(F_u,E_3)E_3
$$
and
$$
F_v=g_{\varepsilon}(F_v,E_1)E_1+g_{\varepsilon}(F_v,E_2)E_2+g_{\varepsilon}(F_v,E_3)E_3\,,
$$
where $\{E_1,E_2,E_3\}$ is the global tangent frame to $\s^3_\varepsilon$
defined in \eqref{eq-basis}, then
$$
N=\frac{N_1 E_1+N_2E_2+N_3 E_3}{\sqrt{N_1^2+N_2^2+N_3^2}}\,,
$$
where
\begin{equation}\label{eq-ni-normal}
\begin{cases}
N_1=g_{\varepsilon}(F_u,E_2)g_{\varepsilon}(F_v,E_3)-g_{\varepsilon}(F_u,E_3)g_{\varepsilon}(F_v,E_2),\\
N_2=g_{\varepsilon}(F_u,E_3)g_{\varepsilon}(F_v,E_1)-g_{\varepsilon}(F_u,E_1)g_{\varepsilon}(F_v,E_3),\\
N_3=g_{\varepsilon}(F_u,E_1)g_{\varepsilon}(F_v,E_2)-g_{\varepsilon}(F_u,E_2)g_{\varepsilon}(F_v,E_1)\,.
\end{cases}
\end{equation}
A long but straightforward computation (that can be also made using a software of symbolic computations) gives
$$
N_1=1/2 (\alpha_1 - \alpha_2) \sqrt{g_{11}} \sqrt{g_{33}}
  \sin(2\xi_1)\sin(\alpha_1 u-\alpha_2 u+\xi_2-\xi_2)( -\xi'+\xi_2'+\xi_3')\,.
$$
Now case (ii) occurs if and only if $\xi_1=\cst=k \pi/2,\, k\in\z$, or if $\xi_1=\cst\neq k \pi/2,\, k\in\z$ and $-\xi'+\xi_2'+\xi_3'=0$. In both cases $N_1=0$ and this implies that $g_{\varepsilon}(N,J_1F)=\varepsilon g_{\varepsilon}(N,E_1)=0$, i.e. the Hopf vector field is tangent to the surface. Thus we have proved that $\xi=\cst$. \\

Now, using \eqref{eq-definition-Fv}, we find
that  $g_{\varepsilon}(F_v,J_1F/\varepsilon)=a \sin^2\vartheta$ and similarly
$g_{\varepsilon}(F_v,F_u)=a \sin^2\vartheta$. This implies that
$$
g_{\varepsilon}(F_v,J_1F)-\varepsilon g_{\varepsilon}(F_v,F_u)=0\,.
$$
Using again a software of symbolic computations one can easily compute
$g_{\varepsilon}(F_v,J_1F)-\varepsilon g_{\varepsilon}(F_v,F_u)$, when $F(u,v)=A(v)\beta(u)$, and find
$$
0=g_{\varepsilon}(F_v,J_1F)-\varepsilon g_{\varepsilon}(F_v,F_u)
=-\varepsilon \cos\vartheta\sqrt{B}[\cos^2(\xi_1(v)) \,\xi_{2}'(v)-\sin^2(\xi_1(v))\, \xi_{3}'(v)].
$$
Since $\vartheta\neq\pi/2$ we conclude that condition \eqref{eq-alpha123} is satisfied.\\

The converse of the theorem can be proved in the following way.  Let $F(u,v)=A(v)\beta(u)$ be a parametrization of a surface in $\s^3_{\varepsilon}$ with $\beta(u)$ given as in \eqref{eq-beta-main-theorem} and $A(v)=A(\xi(v),\xi_1(v),\xi_2(v),\xi_3(v))$ with $\xi=\cst$ and $\xi_1(v),\xi_2(v),\xi_3(v)$ satisfying \eqref{eq-alpha123}. Then the $N_i$'s of the normal vector field described in \eqref{eq-ni-normal} become (again to perform this computation it is better to use a software of symbolic computations and note we have used that $g_{11}=1-g_{33}$):

\begin{equation}\label{eq-ni-normal-viceversa}
\begin{cases}
N_1=(1/\varepsilon) B \sqrt{(1 - g_{33}) g_{33}} ( 2 g_{33}-1)\; \zeta,\\
N_2=2B(g_{33}-1) g_{33} \sin[(2 B u)/\varepsilon + \alpha]\; \zeta,\\
N_3=2B(g_{33}-1) g_{33} \cos[(2 B u)/\varepsilon + \alpha]\; \zeta\,,
\end{cases}
\end{equation}
where
\begin{align*}
\zeta=
2 \xi_1'(v) \cos& \left(\frac{2 B (1-2 {g_{33}})
   u}{\varepsilon}-\xi_2+\xi_3\right)\\
   &-\sin (2\xi_1) \left(\xi_2'+\xi_3'\right) \sin \left(\frac{2 B (1-2 {g_{33}}) u}{\varepsilon}-\xi_2+\xi_3\right)\,.
\end{align*}

Finally, replacing the values of $B$ and of $g_{33}$, we obtain
\begin{equation}\label{final-count-theta}
\frac{g_{\varepsilon}(N,E_1)^2}{g_{\varepsilon}(N,N)}=\frac{N_1^2}{{N_1^2+N_2^2+N_3^2}}=\cos^2\vartheta\,,
\end{equation}
which implies that $F(u,v)$ defines a helix surface with constant angle $\vartheta$.
\end{proof}

\begin{remark}
The geodesic $\beta$ of the torus $\s^1(\sqrt{g_{11}})\times\s^1(\sqrt{g_{33}})\subset \s^3$ in Theorem~\ref{teo-principal} has slope
$$
m=\frac{\alpha_2}{\alpha_1}=\frac{\sqrt{B}-\varepsilon\cos\vartheta}{\sqrt{B}+\varepsilon\cos\vartheta}
$$
that, for fixed $\varepsilon>0$, varying $\vartheta\in(0,\pi/2)$  it can assume all possible values in $(0,1)$.
\end{remark}

\begin{example}\label{ex-mainteo}
We shall now find an explicit example of a  1-parameter family $A(v)$ as in Theorem~\ref{teo-principal}.
Since $\xi=\cst$ and $\xi_1(v),\xi_2(v),\xi_3(v)$ are solutions of \eqref{eq-alpha123}
we can take
$$
\xi=\pi/2\,,\quad \xi_1=\pi/4\,,\quad \xi_2(v)=\xi_3(v)\,.
$$
Then $A(v)$ becomes
$$
A(v)=\frac{1}{\sqrt{2}}\begin{pmatrix}
  \cos \xi_2(v) & -\sin \xi_2(v) & \cos \xi_2(v) & -\sin \xi_2(v) \\
 \sin \xi_2(v)& \cos \xi_2(v) & \sin \xi_2(v) & \cos \xi_2(v)) \\
 -\cos \xi_2(v) & -\sin \xi_2(v) & \cos \xi_2(v) & \sin \xi_2(v) \\
 \sin\xi_2(v) & -\cos \xi_2(v) & -\sin \xi_2(v) & \cos \xi_2(v)
\end{pmatrix}\,.
$$
Using the notation of Theorem~\ref{teo-principal} the map
$$
F(u,v)=A(v) \beta(u)\,,
$$
gives an explicit immersion of a helix surface into the Berger sphere. In Figure~\ref{fig-example}
we have plotted the stereographic projection in $\r^3$ of the surface parametrized by $F$ in the case $\varepsilon=1$ and the constant angle is $\vartheta=\pi/4$.
\end{example}

\begin{figure}[h!]
\centering
\includegraphics[width=.4\textwidth]{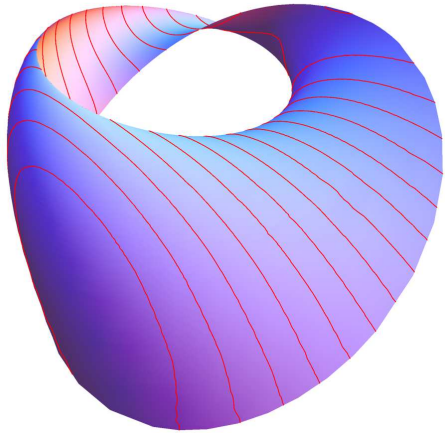}

\caption{\small Stereographic projection in $\r^3$ of the helix surface, of constant angle $\pi/4$, in $\s^3$ given by $F$ in Example~\ref{ex-mainteo}.}
\label{fig-example}
\end{figure}

\end{document}